\documentclass[11pt,reqno]{amsart}
\usepackage{amsmath}
\usepackage{amssymb, amsfonts, amsmath, amsthm, amscd}
\usepackage{mathrsfs}
\usepackage{color, graphics}
\usepackage[all]{xy}
\usepackage{hyperref}
\usepackage{pgf,tikz}
\usepackage{mathrsfs}
\usetikzlibrary{arrows}
\usepackage{amsxtra}
\makeatletter
\renewcommand{\@secnumfont}{\bfseries}
\renewcommand{\section}{\@startsection{section}{1}%
  \z@{.7\linespacing\@plus\linespacing}{.5\linespacing}%
  {\normalfont\bfseries\centering}}
\makeatother
\theoremstyle{plain}
\newtheorem{Theorem}{Theorem}[section]
\newtheorem{Lemma}[Theorem]{Lemma}

\newtheorem{Corollary}[Theorem]{Corollary}
\newtheorem{Remark}[Theorem]{Remark}

\theoremstyle{definition}

\newtheorem{Example}[Theorem]{Example}

\numberwithin{equation}{section}

\begin{document}

\title[Linear series on a curve bridged  by a chain of elliptic curves]{Linear series on a curve of compact type \\
bridged  by a chain of elliptic curves}

\author[Y. Choi]{Youngook Choi}
\address{Department of Mathematics Education, Yeungnam University, 280 Daehak-Ro, Gyeongsan, \hfill \newline\texttt{}
 \indent  Gyeongbuk, 38541, Republic of Korea}
\email{ychoi824@yu.ac.kr}

\author[S. Kim]{Seonja Kim}
\address{Department of Electronic Engineering,
Chungwoon University,  Sukgol-ro, Michuhol-gu,\hfill \newline\texttt{}
 \indent  Incheon,  22100,  Republic of Korea}
\email{sjkim@chungwoon.ac.kr}

\thanks{The first author was supported by Basic Science Research Program through the National Research Foundation of Korea(NRF) 
funded by the Ministry of Education(2019R1I1A3A01055643).  The second  author This work was supported 
by the National Research Foundation of Korea(NRF) grant funded by the Korea government(MSIT) (2019R1F1A1058248).}

\subjclass[2010]{14H10, 14H51, 14J10}

\keywords{algebraic curve, limit linear series, Brill-Noether locus, moduli of curves}



\begin{abstract}

  In the present paper  we  investigate conditions for the non-existence of a limit linear series on   a curve of  compact type such that  two smooth curves   are bridged  by a chain of  two elliptic curves. Combining this work with  results on the existence of a smoothable limit linear series on such a curve, we show relations among Brill-Noether loci of codimension at most  two in the moduli space of complex curves.  Specifically,  Brill-Noether loci of codimension two have  mutually distinct supports.
  \end{abstract}

\maketitle
\section{Introduction}
Let $\mathcal M_g$ be the moduli space of smooth complex curves of genus $g$. It is well known that a general curve of genus $g$ has no   linear series $g^r_d$ whose Brill-Noether number $\rho(g,r,d)(:=g-(r+1)(g-d+r))$ is negative.   The Brill-Noether locus $\mathcal M^r_{g,d}$ is defined by  the sublocus of $\mathcal M_g$ whose elements represent curves possessing a linear series $g^r_d$. 
  We also consider  the closure $\overline{\mathcal M}^r_{g,d}$  of $\mathcal M^r_{g,d}$ in the moduli space $\overline{\mathcal M}_g$ of stable curves of genus $g$. We will  call   $\overline{\mathcal M}^r_{g,d}$ a Brill-Noehter locus as well.  In this work, we study  relations among Brill-Noether loci by investigating existence/non-existence of  a smoothable limit linear series on a 
TCBE$(g_1, g_2; 2, t)$ curve in $\overline{\mathcal M}_g$.  Here,  a TCBE$(g_1, g_2; 2, t)$ curve is defined by  a curve of  compact type such that  two smooth curves $Y_1, Y_2$  are bridged  by a chain of  two elliptic curves $E_1, E_2$  as in Figure \ref{figure2},  which satisfies  that for each $i=1,2$
\begin{itemize} 
\item[$(i)$]  $(Y_i, p_i)$ is a general one pointed curve of genus $g_i \geq 2\, ,$
\item[$(ii)$]  $t = \mbox{min} \{ {\tilde t} \in {\mathbb Z}^{>0} ~ |~ \mathcal O _{E_i} ( {\tilde t} ( p_i -q)) \simeq \mathcal O _{E_i} \}.$
\end{itemize}
\begin{figure}[ht]
\begin{center}
\definecolor{uuuuuu}{rgb}{0.26666666666666666,0.26666666666666666,0.26666666666666666}
\begin{tikzpicture}[line cap=round,line join=round,>=triangle 45,x=0.6cm,y=0.54cm]
\clip(-1.5128,-3.5752) rectangle (14.5934,2.4574);
\draw (-1.3746,0.9724) node[anchor=north west] {$E_1$};
\draw (0.3532,-2.5588) node[anchor=north west] {$Y_1$};
\draw (5.0848,0.996) node[anchor=north west] {$E_2$};
\draw (12.7282,-2.5588) node[anchor=north west] {$Y_2$};
\draw (1.4006,0.686) node[anchor=north west] {$p_1$};
\draw (12.4068,0.605) node[anchor=north west] {$p_2$};
\draw [shift={(8.417346487235307,3.201197506431824)}] plot[domain=3.9542093264964904:5.418856023554369,variable=\t]({1.*4.200033317054753*cos(\t r)+0.*4.200033317054753*sin(\t r)},{0.*4.200033317054753*cos(\t r)+1.*4.200033317054753*sin(\t r)});
\draw [shift={(12.962227167630058,-2.1928391136801535)}] plot[domain=0.8035042551907858:2.2616924551374797,variable=\t]({1.*2.853652053617564*cos(\t r)+0.*2.853652053617564*sin(\t r)},{0.*2.853652053617564*cos(\t r)+1.*2.853652053617564*sin(\t r)});
\draw [shift={(0.32984722753346074,1.5112652007648184)}] plot[domain=-1.045064318665168:1.10060559185201,variable=\t]({1.*1.4880109321528001*cos(\t r)+0.*1.4880109321528001*sin(\t r)},{0.*1.4880109321528001*cos(\t r)+1.*1.4880109321528001*sin(\t r)});
\draw [shift={(2.134408888888889,-1.2952014814814814)}] plot[domain=2.1789680993673075:3.773919852983025,variable=\t]({1.*1.8513617980666195*cos(\t r)+0.*1.8513617980666195*sin(\t r)},{0.*1.8513617980666195*cos(\t r)+1.*1.8513617980666195*sin(\t r)});
\draw [shift={(14.384335474860338,1.4570480446927365)}] plot[domain=2.138274432124917:3.9751072881552245,variable=\t]({1.*1.4365120087087788*cos(\t r)+0.*1.4365120087087788*sin(\t r)},{0.*1.4365120087087788*cos(\t r)+1.*1.4365120087087788*sin(\t r)});
\draw [shift={(11.872784723195517,-0.9959667133847234)}] plot[domain=-0.7994205597185902:0.7322189708866035,variable=\t]({1.*2.078566939251402*cos(\t r)+0.*2.078566939251402*sin(\t r)},{0.*2.078566939251402*cos(\t r)+1.*2.078566939251402*sin(\t r)});
\draw (6.8088,-0.8112) node[anchor=north west] {$q$};
\draw [shift={(1.3422919847328243,-1.6999770992366408)}] plot[domain=0.7818969516192444:2.3319197101086147,variable=\t]({1.*2.490364111054216*cos(\t r)+0.*2.490364111054216*sin(\t r)},{0.*2.490364111054216*cos(\t r)+1.*2.490364111054216*sin(\t r)});
\draw [shift={(5.848945217391304,2.867576521739131)}] plot[domain=3.9401793643809606:5.609436216380334,variable=\t]({1.*3.9264258250192947*cos(\t r)+0.*3.9264258250192947*sin(\t r)},{0.*3.9264258250192947*cos(\t r)+1.*3.9264258250192947*sin(\t r)});
\end{tikzpicture}
\end{center}
\vspace{-0.4cm}
\caption{\label{figure2}}
\end{figure}

\noindent{More generally,}   a TCBE$(g_1, g_2; n, t)$ curve is defined  in  \cite{SK} by the same way as  a TCBE$(g_1, g_2; 2, t)$ curve except that   the chain of elliptic curves has  length $n$.  

The second author \cite{SK} gave    conditions on $t$ in terms of $g,r,d$ under which a TCBE$(g_1, g_2; n, t)$ curve in $\overline{\mathcal M}_g$ carries a smoothable limit linear series $g^r_d$.
 Along the line,  the focus of this work lies on finding conditions on $t$ for   a TCBE$(g_1, g_2; 2, t)$ curve not to admit a $g^r_d$. This work combined with the existence conditions given in \cite{SK} will enable us to see some relations among Brill-Noether loci  in  $\mathcal M_g$ or 
 $\overline{\mathcal M}_g$. More precisely, we describe  differences and intersections among   Brill-Noether loci  by using the family of TCBE$(g_1, g_2; 2, t)$ curves in $\overline{\mathcal M}_g$, whose dimension equals $3g-8$. For convenience sake, we define the following:
 $${\Delta}  ^{BE} _g (g_1, g_2; 2,t) : = \{ [C] \in \overline {\mathcal M} _g ~|~ C \mbox{ is a  TCBE}(g_1, g_2; 2,t ) \mbox{ curve, } g_1\geq g_2\} .$$
 
 In case    $\rho (g, r, d) =-1$,  $\overline{\mathcal M}^r_{g,d}$ is  irreducible and of codimension one in $\overline{\mathcal M}_g$ \cite{EH1989, St}, which is called a  Brill-Noether divisor. Relationships among  Brill-Noether divisors  inside  $\mathcal M_g$ are useful for  dealing with  the birational geometry of $\mathcal M_g$  \cite{EH1987-2, Farkas, Farkas_home}.  It was known   that any two Brill-Noether divisors  have mutually  distinct supports in  $\mathcal M_g$ \cite{CKK1, CKK2}.
On the one hand,  every    component of $\mathcal M^r_{g,d}$ has codimension at most $-\rho (g, r, d)$ in $\mathcal M_{g}$; further, in case  $-3 \leq\rho (g,r,d) \leq -1$, each component of  $\mathcal M^r_{g,d}$ is of codimension $-\rho (g, r, d)$ \cite{Edidin, EH1989, St}.

The objective  of this study  is to  find  relations between  a locus $\pmb{\bigtriangleup}^{BE} _g (g_1, g_2;2,t)$   and a Brill-Noether locus  $\overline{\mathcal M}^r_{g,d}$. Theorem \ref{mathm} in the present paper shows the following:

{\it Let $\rho (g,r,d)=\rho  <0$ and $t\geq 4$.  Then we have  
 \begin{equation*}
 \pmb{\bigtriangleup}^{BE} _g (g_1, g_2;2,t) \cap \overline{\mathcal M}^r_{g,d} =\emptyset,\end{equation*} 
  when
\begin{eqnarray*}   t  \geq \begin{cases}   g-d+2r  +(g_1 -g_2) +\delta_{g_1, g_2} &\mbox{  in  case  }\rho = -1, \\
\frac{2}{-\rho }\big( \, g-d+2r -2 +(g_1 -g_2) +\delta_{g_1, g_2}\big)
&\mbox{ in  case } \rho \leq -2.
\end{cases}
\end{eqnarray*}
Here $\delta_{g_1, g_2}$ denotes the  Kronecker delta.}

This result is a kind of counter part of Theorem 1.1 in \cite{SK} for the case $n=2$: under the hypotheses $\rho (g,r,d) =-2+h$ and $g_1 -h\geq g_2 \geq 2$ with $h=0,1$,  we get 
 \begin{equation*}
 {\Delta} ^{BE}_g (g_1,g_2; 2,t) \subset \overline{\mathcal M}_{g, d}^r,
 \end{equation*} 
 if
   \begin{equation} \label{ranget}
   \begin{cases} 
   2\leq t\leq g_2 +h+1 , \ \ t\equiv g_1 +1 \ (\mbox{mod }2)  &\mbox{ in case  } r =1\\
  t_{min}  \leq t \leq g-d+2r -2+h  \ \ &\mbox{  in case  } r \geq 2
    \end{cases}
 \end{equation}
where \begin{equation} \label{tmin} t_{min} := \begin{cases}  r+2+(g_1 -g_2) 
 \ \mbox{ if  } r   \mbox{ is even,}\\
r +2 +  \frac{g_1 -g_2+h  }{2}  \ \ \ \ \ \mbox{ if  both } r   \mbox{  and } d+2 + \frac{g_1 -g_2+h}{2} \mbox{ are odd,}\\
 r +2  +1 + \frac{ g_1 -g_2+h }{2}     \ \mbox{ if  } r   \mbox{ is odd  and } d+2 +\frac{g_1 -g_2+h}{2} \mbox{  is even.}
\end{cases}  
\end{equation}

It is interesting that, for the case $\rho  (g,r,d)=-2$ and $r\geq 2$,  the above two results on the non-existence/existence of  a smoothable limit $g^r_d$ tell  the sharpness of each other when $(g_1,g_2) = (\lceil \frac{g-2}{2}\rceil, \lfloor \frac{g-2}{2}\rfloor),$ where  $\lceil \frac{g-2}{2}\rceil$ (resp. $\lfloor \frac{g-2}{2}\rfloor$) is the smallest  integer greater (resp. the largest integer less) than or equal to $\frac{g-2}{2}$.
  For such a $(g_1,g_2)$,  the two results have only one gap for sharpness in case $\rho  (g,r,d)=-1$ and $r\geq 2$.   
From this observation we get  corollaries as follows:
\begin{enumerate}
 \item   
 Brill-Noether loci of Brill-Noether number $\rho =-2$ have  mutually distinct supports  in $\mathcal M_g$
(see  Corollary  \ref{mcoro4}),
 \item   $ {\mathcal M}^r_{g,d}$ with $\rho (g,r,d)  =-2$ is not contained in any Brill-Noether divisor ${\mathcal M}^s_{g,e}$ with  $e-2s\geq d-2r+3$ (see Corollary  \ref{macor2}),
\item if $X$ is  a general plane curve of degree $d$ and genus $g$ with  $34\leq g\leq \frac{3d-4}{2}$, then  its   smooth model $\tilde{X}$ does not admit  $g^s_e$   with  $s\geq 2$ and $\rho (g,s,e)<0$ except  $ g^2_d$ and   $ K_{\tilde{X} }-g^2_d$
(see Remark \ref{rmknet}).
\end{enumerate}

Considering the above, one can notice that results on the non-existence/existence of a limit $g^r_d$
 would play a role in studying not only relations among Brill-Noether loci but also embeddings of individual curves belonging to a prescribed  Brill-Noether locus of codimension at most two.
In  Example \ref{exam34} we consider  $\overline{\mathcal M}_{34}$ in which the  Brill-Noether loci  with $\rho =-2$ are $\overline{\mathcal M} ^1_{34, 17}, \  \overline{\mathcal M} ^2_{34, 24}, \ \overline{\mathcal M} ^3_{34, 28} , \  \overline{\mathcal M} ^5_{34, 33}$ and  the  Brill-Noether divisor is  $\overline{\mathcal M} ^4_{34, 31}$. It 
presents their relations as follows:
\begin{enumerate}
\item[$(i)$] $\pmb{\bigtriangleup} ^{BE}_{34} (16,16;2,9)  \subset  \overline{\mathcal M} ^1_{34, 17} \cap \overline{\mathcal M} ^2_{34, 24} \cap \overline{\mathcal M} ^3_{34, 28} \cap \overline{\mathcal M}^5_{34, 33}  \cap \overline{\mathcal M} ^4_{34, 31}$,\\
\item[$(ii)$] $\bigcup\limits_{t=13,15,17}\pmb{\bigtriangleup} ^{BE}_{34} (16,16;2,t)  \subset  \overline{\mathcal M} ^1_{34, 17} -(\overline{\mathcal M}^2_{34, 24} \cup \overline{\mathcal M} ^3_{34, 28}\cup \overline{\mathcal M}^5_{34, 33}\cup \overline{\mathcal M} ^4_{34,31}  )$.
\end{enumerate}
 These are not trivial since  $ \pmb{\bigtriangleup} ^{BE}_{34} (16,16;2,9) $ is of codimension five and  Brill-Noether loci of codimension two have mutually distinct supports (see Corollary  \ref{mcoro4}). The result $(ii)$  tells how a general   curve $X$ in  ${\mathcal M} ^1_{34, 17}$ is embedded  by $|K_X -g^1_{17}|$. For instance, the relation  $\pmb{\bigtriangleup} ^{BE}_{34} (16,16;2,t) \subset   \overline{\mathcal M} ^1_{34, 17} -  \overline{\mathcal M} ^2_{34, 24}$  implies that if  $L_{\eta}=g^1_{17}$ on a smooth curve $C_{\eta}$ is a  smoothing of   $L=g^1_{17}$ on $C \in\pmb{\bigtriangleup} ^{BE}_{34} (16,16;2,t)$ then the residual $|K_{C_{\eta}}- L_{\eta}|$  embeds  $C_{\eta}$ into ${\mathbb P} H^0 (C_{\eta},K_{C_{\eta}} - L_{\eta})$  such that any seven  points of the curve are in general position.
 \vspace{0.4cm}
 
 \noindent
{\bf Acknowledgements}
We thank KIAS for the warm hospitality when we were associate members in KIAS .

 \vspace{0.4cm}

\section{Preliminaries}
In this section, we review definitions and theorems on limit linear series in   \cite{EH1986, EH1987-2} which will be used to verify   the existence/non-existence of a smoothable  limit linear series $g^r_d$ on a $ TCBE(g_1, g_2; 2,t)$ curve.
If $C$ is a curve of compact type, {\it a (crude) limit} $g^r_d$ on $C$ is a collection of ordinary linear
series $L=\{L_Y\in G^r_d(Y)\ |\ Y\subset C {\text{ is a component}}\}$ satisfying the
 compatibility condition: if $Y$ and $Z$ are components of $C$ with $\{p\}=Y\cap Z$, then
\begin{equation}\label{limit}
a^{L_Y}_j(p)+a^{L_Z}_{r-j}(p)\geq d \   {\text{ for }} 0\le j\le r,
\end{equation}
where $\{a^L_j(p): a^L_0(p)< a^L_1(p)< \cdots < a^L_r(p)\}$ is the vanishing
sequence  of $L$ at $p$. Recall that the sequence $(\alpha^L_0(p),\cdots,\alpha^L_r(p))$ with $\alpha^L_j(p):=a^L_j(p)-j$ is called {\it the ramification sequence of $L$ at $p$}. If the equality in  (\ref{limit}) holds everywhere, $L$ is said to be  {\it a refined limit}
$g^r_d$. The linear series $L_Y\in G^r_d(Y)$ is called {\it the $Y$-aspect of $L$}.
A limit linear series $g^r_d=L$ on $C$   is said to be {\it smoothable} if there is a flat family $\pi : \mathcal C \to B$, $B=Spec(R),$ for a discrete valuation ring $R$ with $C=C_0$ and a $g^r_d=L_\eta$ on $C_\eta$ for a generic point $\eta \in B$ whose limit is $L$ on $C=C_0$.

Therefore, in order to study  problems related to the  existence/non-existence of
 a smoothable  limit $g^r_d$ on a  TCBE$(g_1,g_2;2,t)$ curve, we first examine conditions for  the existence of an ordinary $g^r_d$ satisfying specific vanishing conditions on each  component of the   TCBE$(g_1,g_2;2,t)$ curve. Thus  results on linear series with a prescribed vanishing sequence  are essential in this work.

\begin{Theorem}[\cite{EH1987-2}, (1.2) Proposition]\label{1987-2}
A general pointed curve $(C,q)$ of genus $g$ possesses a $g^r_d$ with ramification sequence $(\alpha_0,\cdots, \alpha_r)$ at $q$ if and only if
\begin{equation}\label{ineqdim}
\sum^r_{j=0}(\alpha_j+g-d+r)_+\le g,\end{equation}
where  $(\alpha_j+g-d+r)_+:=\max\{0,\alpha_j+g-d+r\}.$
\end{Theorem}
\begin{Remark}\label{dimpropergen} Let $( C, p) $ be a  general pointed curve of genus $g$.  Then  Theorem 4.5 in \cite{EH1986} tells that every component of
\begin{equation*}
 \{ L \in G^r_d (C) ~|~ \alpha _j ^L (p) = \alpha _j ,~ j=0, \cdots , r \}
\end{equation*}
has dimension  $ \rho (g,r,d) -\sum _{j=0}^r \alpha_j$, which equals $g-\sum^r_{j=0}(\alpha_j+g-d+r)$.
\end{Remark}
 
 For an $m$-pointed curve $(Y, p_1, \dots, p_m)$ of genus $g_Y$,  {\it  the adjusted  Brill-Noether number of $L_Y$ with respect to} $\{ p_1, \dots, p_m\}$ is defined by 
\begin{equation}\label{adbn}
\rho (L_Y, p_1, \dots, p_m ):= \rho (g_Y,r,d) - \sum _{i=1}^m  \sum _{j=0}^r\alpha^{L_Y}_j  (p_i).
\end{equation}
The following lemma shows a relation  between the additivity of  adjusted  Brill-Noether numbers of aspects and  the  refinedness  of a limit $g^r_d$.

\begin{Lemma}\label{additive0}
Let  $C:=Y_1 \cup \cdots \cup Y_n$   be a  curve of compact type  of genus $g$ which is a chain of smooth curves $Y_1, \cdots,  Y_n$   with $ Y_i \cap Y_{i+1} =p_i$  for $i= 1,\cdots, n-1$, as in Figure \ref{figure12}.
Assume that $C$ possesses a limit linear series $g^r_d := \{ L_{Y_1},  \cdots , L_{Y_n} \}$.
 Then,
 $$\rho ( L_{Y_1}, p_1 ) +\sum _{i=2} ^{n-1}  \rho ( L_{Y_i}, p_{i-1}, p_i ) +\rho ( L_{Y_n}, p_{n-1} ) = \rho (g,r,d) -\sum_{i=1}^{n-1} \sum_{j=0}^{r}  \eta _{ij},$$
 where
$  \eta _{ij} := a_j ^{L_{Y_i}}( p_i)+ a _{r-j} ^{L_{Y_{i+1}}}( p_i)-d$.
\end{Lemma}
\begin{figure}[ht]
\begin{center}
\definecolor{uuuuuu}{rgb}{0.26666666666666666,0.26666666666666666,0.26666666666666666}
\begin{tikzpicture}[line cap=round,line join=round,>=triangle 45,x=0.38cm,y=0.38cm]
\clip(-1.8192,-2.966) rectangle (29.827,1.656);
\draw [shift={(1.1904820105820104,-1.0081375661375664)}] plot[domain=0.7015005375176029:2.3251110471055663,variable=\t]({1.*2.322001365871798*cos(\t r)+0.*2.322001365871798*sin(\t r)},{0.*2.322001365871798*cos(\t r)+1.*2.322001365871798*sin(\t r)});
\draw [shift={(5.5563144486692,2.9832300380228136)}] plot[domain=3.907468126624897:5.2959953917966915,variable=\t]({1.*3.5962840423788065*cos(\t r)+0.*3.5962840423788065*sin(\t r)},{0.*3.5962840423788065*cos(\t r)+1.*3.5962840423788065*sin(\t r)});
\draw [shift={(6.983003614457831,1.350860642570281)}] plot[domain=3.404586706334768:5.720937158008401,variable=\t]({1.*2.658203454761023*cos(\t r)+0.*2.658203454761023*sin(\t r)},{0.*2.658203454761023*cos(\t r)+1.*2.658203454761023*sin(\t r)});
\draw [shift={(11.159353692614769,-1.3394774451097815)}] plot[domain=0.8021221774449528:2.557777401962595,variable=\t]({1.*2.309962707201369*cos(\t r)+0.*2.309962707201369*sin(\t r)},{0.*2.309962707201369*cos(\t r)+1.*2.309962707201369*sin(\t r)});
\draw [shift={(18.24482338257769,-1.6167074375955195)}] plot[domain=0.6686200153256092:2.231298679950041,variable=\t]({1.*2.423130164757994*cos(\t r)+0.*2.423130164757994*sin(\t r)},{0.*2.423130164757994*cos(\t r)+1.*2.423130164757994*sin(\t r)});
\draw [shift={(22.996513592233015,2.3519203883495163)}] plot[domain=3.8549312636883593:5.462806160162388,variable=\t]({1.*3.7693514561807735*cos(\t r)+0.*3.7693514561807735*sin(\t r)},{0.*3.7693514561807735*cos(\t r)+1.*3.7693514561807735*sin(\t r)});
\draw [shift={(25.277962798634814,1.13836257110353)}] plot[domain=3.464058498175733:5.7383840912891,variable=\t]({1.*2.8846465367389906*cos(\t r)+0.*2.8846465367389906*sin(\t r)},{0.*2.8846465367389906*cos(\t r)+1.*2.8846465367389906*sin(\t r)});
\draw [shift={(30.30336872998933,-1.8508918890074666)}] plot[domain=1.191933773174612:2.6129613223844443,variable=\t]({1.*2.9627957756383947*cos(\t r)+0.*2.9627957756383947*sin(\t r)},{0.*2.9627957756383947*cos(\t r)+1.*2.9627957756383947*sin(\t r)});
\draw (-1.2678,1.9802) node[anchor=north west] {$Y_1$};
\draw (3.529,1.819) node[anchor=north west] {$Y_2$};
\draw (13.172,0.6356) node[anchor=north west] {$\cdots\cdots$};
\draw (15.2616,1.8992) node[anchor=north west] {$Y_{n-1}$};
\draw (21.5268,1.5082) node[anchor=north west] {$Y_n$};
\draw (4.2574,-0.5748) node[anchor=north west] {$p_1$};
\draw (22.5178,-1.3912) node[anchor=north west] {$p_{n-1}$};
\end{tikzpicture}
\end{center}
\vspace{-0.4cm}
\caption{\label{figure12}}
\end{figure}
\begin{proof} The definition of an adjusted Brill-Noether number in \eqref{adbn} yields that
\begin{eqnarray*}
&&\rho ( L_{Y_1}, p_1 ) +\sum _{i=2} ^{n-1}  \rho ( L_{Y_i}, p_{i-1}, p_i ) +\rho ( L_{Y_n}, p_{n-1} ) \\
&=& \{  -rg(Y_1) -\sum _{j=0}^r(r+\alpha _j^{L_{Y_1}}( p_1) -d)\} +
 \{ - rg (Y_n)-\sum _{j=0}^r(r+\alpha_j ^{L_{Y_n}}( p_{n-1}) -d) \}\\
&&+\sum_{i=2}^{n-1} \{- rg(Y_i)-\sum _{j=0}^r(r+\alpha_j ^{L_{Y_i}}( p_{i-1})
+\alpha_j ^{L_{Y_i}}( p_{i}) -d) \}  \\
&=& \{  -rg(Y_1) -\sum _{j=0}^r(j+a _{r-j}^{L_{Y_1}}( p_1) -d)\} +
 \{ - rg (Y_n)-\sum _{j=0}^r(r-j+a_{j} ^{L_{Y_n}}( p_{n-1}) -d) \}\\
&&+\sum_{i=2}^{n-1} \{ -rg(Y_i)-\sum _{j=0}^r(a_j ^{L_{Y_i}}( p_{i-1})
+a_{r-j} ^{L_{Y_i}}( p_{i}) -d) \}  \\
&=& -rg-(r+1)(r-d)-\sum_{i=1}^{n-1} \sum _{j=0}^r(a_{r-j} ^{L_{Y_i}}( p_i)+ a _{j} ^{L_{Y_{i+1}}}( p_i)-d)\\
 &=& \rho(g,r,d) -\sum_{i=1}^{n-1} \sum_{j=0}^{r}  \eta _{ij},
\end{eqnarray*}
since $g(Y_1) +\cdots + g(Y_n) =g$.  Thus the result follows.
\end{proof}

Concerning  adjusted Brill-Noether numbers on an elliptic curve, we get an equation as in what follows.
\begin{Lemma}[\cite{SK}, Lemma 2.5]\label{zeros}
Let $L_E$ be a $g^r_d$ on an elliptic curve $E$ and
$\nu _{j} :=d-(a_j^{L_E} (p) + a_{r-j}^{L_E} (q)) $ for $ p,q \in E$.
 Then we have
 \begin{equation}\label{zero}
\#  \{j \ | \ \nu_{j} =0, j=0, \cdots, r \}= -\rho (L_E , p, q) +1 +\sum _{j=o} ^r (\nu _j -1)_{+}
\end{equation}
where $(\nu _j -1)_{+} := \mbox{max} \{\nu _j -1 , 0\}$.
\end{Lemma}


\section{Limit linear series on a TCBE$\pmb{(g_1, g_2;2,t)}$ curve}
In this section we investigate  necessary conditions on $t$ for a TCBE$(g_1,g_2;2,t)$ curve to admit  a limit linear series   $g^r_d$ with $\rho (g,r,d)<0$, whereas sufficient conditions  for  the existence of a smoothable  limit  $g^r_d$ with $\rho (g,r,d)=-1, -2$ were given in  Theorem 1.1 in \cite{SK}.  Our result combined with  Theorem 1.1 \cite{SK}  gives rise to  some relations  among  Brill-Noether loci corresponding to $\rho =-1,-2$  in the moduli space $\mathcal M_g$.
In particular,  it will be shown that Brill-Noether loci of codimension two have mutually distinct supports in $\mathcal M_g$  as in the case of Brill-Noether divisors \cite{CKK1, CKK2}.

Before going to the theorem, we demonstrate some lemmas for the proof of our main theorem.

\begin{Lemma}\label{additive}
Let  $C:=Y_1 \cup \cdots \cup Y_n$   be the same curve of compact type  of genus $g$  as  in Lemma \ref{additive0}.
Assume that $C$ possesses a limit linear series $g^r_d := \{ L_{Y_1},  \cdots , L_{Y_n} \}$.
 Then,  for each $i=1, \dots,  n-1$
\begin{enumerate}
   \item
$a_r ^{L_{Y_{i+1}}}(p_i) \leq g(Y_i) +r +\rho ( L_{Y_i}, p_i ) +\eta _{i0}$ \ in case  $(Y_i, p_i) $ is general,
  \item
$a_r ^{L_{Y_{i}}}(p_{i}) \leq g(Y_{i+1}) +r +\rho ( L_{Y_{i+1}}, p_{i}) +\eta _{ir}$   \ in case  $(Y_{i+1}, p_{i}) $ is general,
\end{enumerate}
where
$  \eta _{ij} := a_j ^{L_{Y_i}}( p_i)+ a _{r-j} ^{L_{Y_{i+1}}}( p_i)-d$.
\end{Lemma}
\begin{proof} (1)
 Since $a_0 ^{L_{Y_i}}(p_i)+ a _{r} ^{L_{Y_{i+1}}}( p_i) =d +\eta _{i0}$, the conclusion  $a_r ^{L_{Y_{i+1}}}(p_i) \leq g(Y_i) +r +\rho ( L_{Y_i}, p_i ) +\eta _{i0}$ is equivalent to  $g(Y_i) -d +r+\alpha _0^{L_{Y_i}}( p_i) \geq  -\rho ( L_{Y_i}, p_i )$.
Assume that $g(Y_i) -d +r+\alpha _0^{L_{Y_i}}( p_i) <  -\rho ( L_{Y_i}, p_i )$. Then   the equation  $\rho ( L_{Y_i}, p_i ) = g(Y_i) - \sum _{j=o} ^r (g(Y_i)-d +r+\alpha _j^{L_{Y_i}}( p_i))$ yields  $\sum _{j=1} ^r (g(Y_i) -d+r+\alpha _j^{L_{Y_i}}( p_i)) >g(Y_i),$ which means $\sum _{j=o} ^r (g(Y_i) -d +r+\alpha _j^{L_{Y_i}}( p_i)) _+ >g(Y_i).$ This cannot occur by Theorem \ref{1987-2} since $(Y_i , p_i)$ is general.

(2) This can be shown by the  same arguments as in the proof of (1).
\end{proof}

\begin{Lemma}\label{inequal}
Let $C$ be a TCBE$(g_1, g_2; 2, t)$ curve as in Figure \ref{figure2}.
 Assume that $C$ admits a limit linear series $g^r_d := \{ L_{Y_1},  L_{E_1}, L_{E_2}, L_{Y_2} \}$ with $\rho ( L_{Y_i}, p_i )=:\gamma _i$ for $i=1,2$.  Let  $ \eta _{ij}:= a_j ^{L_{Y_i}}( p_i) +a_{r-j} ^{L_{E_i}}( p_i) -d$, $\beta _{j} := a_j ^{L_{E_1}}( q) +a_{r-j} ^{L_{E_2}}( q) -d$, $\nu _{ij} :=d-(a_j^{L_{E_i}} (p_i) + a_{r-j}^{L_{E_i}}(q)) $ and  $m _i:=-\rho (L_{E_i}, p_i, q) +\sum _{j=o} ^r (\nu _{ij} -1)_{+}$
for $i=1,2$ and $j=0,\cdots, r$.  
\begin{enumerate}
  \item  If $m_1\ge 1$, then 
\begin{equation}\label{first}
\frac{d+m _1 t}{2} + \lceil \frac{\nu _{1r}}{\nu _{1r} +1}\rceil  \leq g_1 +r +\gamma _1 +\eta_{10},
\end{equation}
or
\begin{equation}\label{second}
\frac{d+m _1 t}{2} -\beta _r- \nu _{2r}  <  g_2 +r +\gamma _2 +\eta_{2r}.
\end{equation}
  \item  If $m_2\ge 1$, then 
\begin{equation*}
\frac{d+m _2  t}{2}  \leq g_2 +r +\gamma _2 +\eta_{2r},
\end{equation*}
or
\begin{equation*}
\frac{d+m _2 t}{2} -\beta _0- \nu _{1r} +\lceil \frac{\nu _{2r}}{\nu _{2r} +1}\rceil   <  g_1 +r +\gamma _1 +\eta_{10}.
\end{equation*}  
\end{enumerate}
\end{Lemma}

\begin{proof}
(1) By Lemma \ref{zeros}, we have $\#  \{j \ | \ \nu_{1j} =0, j=0, \cdots, r \} =m_1+1$. Hence, there is a strictly increasing sequence 
 $\{ j(1,0), \cdots , j(1,m_1) \} \subset \{ 0, \cdots, r\}$ such that for each $j(1,l)$  there  is a
$\sigma _{j(1,l)} $ of $L_{E_1}$ satisfying
\begin{equation*}\mbox{div}(\sigma _{j(1,l)})
=a_{j(1,l)} ^{L_{E_1}} (p_1) p_1+ a_{r- j(1,l) }^{L_{E_1}} (q)q.
\end{equation*}
We will derive equations \eqref{first} and \eqref{second} through comparing   the integers $a_{j(1,0)} ^{L_{E_1}} (p_1)$ and $a_{r-j(1,m_1)} ^{L_{E_1}} (q)$.
First we assume that $a_{j(1,0)} ^{L_{E_1}} (p_1) \geq  a_{r- j(1,m _1) }^{L_{E_1}} (q) $. Using Lemma \ref{additive}, (2), we have 
\[a_{r} ^{L_{E_1}} (p_1) \leq g_1 +r +\gamma _1 +\eta_{10}. \] 
Thus the inequality \eqref{first} can be 
  given when we verify $ a_{r} ^{L_{E_1}} (p_1) \geq \frac{d+m _1 t}{2} + \lceil \frac{\nu _{1r}}{\nu _{1r} +1}\rceil  $.
Since   $$a_{j(1,l)} ^{L_{E_1}} (p_1)p_1 + a_{r- j(1,l) }^{L_{E_1}} (q)q \sim a_{j(1,0)} ^{L_{E_1}} (p_1)p_1 + a_{r- j(1,0) }^{L_{E_1}} (q)q$$ and $m _1  +1 = \#  \{j \ | \ \nu_{1j} =0, j=0, \cdots, r \} \geq 2$, we get  $a_l(p_1-q)\sim 0$ for all $1\le l \le m _1$,
where
$$a_l :=a_{j(1,l)} ^{L_{E_1}} (p_1) - a_{j(1,0)} ^{L_{E_1}} (p_1)=  a_{r- j(1,0) }^{L_{E_1}} (q) - a_{r- j(1,l) }^{L_{E_1}} (q)  \  \mbox{ for }  l=1, \cdots, m _1.$$
This forces  that every $a_l$ is a multiple of the torsion $t$ and hence
$$ a_{m _1} =n_1 t  \  \mbox{ for some integer } n_1\geq m _1, $$
since $a_1<a_2<\cdots<a_{m _1}.$
This yields that
\begin{eqnarray}\label{eq}
\notag  a_{j(1,0)} ^{L_{E_1}} (p_1) +  a_{r- j(1,m _1 ) }^{L_{E_1}} (q) &=& a_{j(1,0)} ^{L_{E_1}} (p_1) + a_{r- j(1,0 ) }^{L_{E_1}} (q)-(a_{r- j(1,0 ) }^{L_{E_1}} (q) -a_{r- j(1,m_1 ) }^{L_{E_1}} (q))\\
  &=& d-n_1 t  ,\ \
\end{eqnarray}
whence
$$ a_{r- j(1,m _1 ) }^{L_{E_1}} (q)\leq \frac{d-n_1 t}{2}$$
by the assumption  $a_{j(1,0)} ^{L_{E_1}} (p_1) \geq  a_{r- j(1,m_1) }^{L_{E_1}} (q) $.
Therefore,
\begin{eqnarray*}
a_{r} ^{L_{E_1}} (p_1) &\geq& a_{j(1,m _1)} ^{L_{E_1}} (p_1) +  \lceil \frac{\nu _{1r}}{\nu _{1r} +1}\rceil \\
& =&  d- a_{r- j(1,m _1 ) }^{L_{E_1}} (q) + \lceil \frac{\nu _{1r}}{\nu _{1r} +1}\rceil \\
&\geq& \frac{d +n_1t}{2} +  \lceil \frac{\nu _{1r}}{\nu _{1r} +1}\rceil  ,
\end{eqnarray*}
since $\nu_{1r} =0$ if and only if ${j(1,m _1)} =r$.
This implies the inequality \eqref{first}  by  $n_1 \geq m _1$.

Now we assume that $a_{j(1,0)} ^{L_{E_1}} (p_1) <  a_{r- j(1,m _1) }^{L_{E_1}} (q) $.
The equation \eqref{eq} yields
$$a_{r- j(1,m _1 ) }^{L_{E_1}} (q)  > \frac{d -n_1t}{2} .  $$
 Therefore we have
\begin{equation*}
a_{r }^{L_{E_1}} (q) \geq a_{r- j(1,0 ) }^{L_{E_1}} (q) = a_{r- j(1,m _1 ) }^{L_{E_1}} (q) +n_1 t  > \frac{d +n_1 t}{2} ,
\end{equation*}
whence
\begin{eqnarray*}
a_{r }^{L_{E_2}} (p_2) &=&d- \nu _{2r} - a_{0} ^{L_{E_2}} (q)\\
&=&- \nu _{2r} -\beta _r +a_{r }^{L_{E_1}} (q) \ \ \mbox{ by  }\beta _r =a_{r }^{L_{E_1}} (q)+a_{0} ^{L_{E_2}} (q) -d\\
&>& \frac{d +m_1 t}{2}    - \nu _{2r} -\beta _r , \  \mbox{ for } n_1 \geq m_1\ .
\end{eqnarray*}
Thus the result $(1)$ is proved by Lemma \ref{additive}.

(2) By Lemma \ref{zeros},  we have a strictly increasing sequence 
 $\{ j(2,0), \cdots , j(2,m_2) \} \subset \{ 0, \cdots, r\}$ such that for each $j(2,l)$  there  is a
$\sigma _{j(2,l)} $ of $L_{E_2}$ satisfying
\begin{equation*}\mbox{div}(\sigma _{j(2,l)})
=a_{j(2,l)} ^{L_{E_2}} (p_2) p_2+ a_{r- j(2,l) }^{L_{E_2}} (q)q.
\end{equation*}
By the same reasoning as in the proof  of (1),  we get 
\begin{equation*}
a_{j(2,0)} ^{L_{E_2}} (p_2) +  a_{r- j(2,m _2 ) }^{L_{E_2}} (q) = d-n_2 t  {\mbox{ for some $n_2 \geq  m_2$}}.
\end{equation*}
 If $a_{j(2,0)} ^{L_{E_2}} (p_2) \geq  a_{r- j(2,m _2) }^{L_{E_2}} (q) -2 \lceil \frac{\nu _{2r}}{\nu _{2r} +1}\rceil $,  then we  obtain 
$ a_{r- j(2,m _2) }^{L_{E_2}} (q)\leq \frac{d-n_2 t}{2}+\lceil \frac{\nu _{2r}}{\nu _{2r} +1}\rceil  .$
This yields $$a_r^{L_{E_2}} (p_2) \geq \frac{d+m_2t}{2}.$$
If $a_{j(2,0)} ^{L_{E_2}} (p_2) <  a_{r- j(2,m _2) }^{L_{E_2}} (q) -2 \lceil \frac{\nu _{2r}}{\nu _{2r} +1}\rceil$, we have $ a_{r- j(2,m _2) }^{L_{E_2}} (q) 
>\frac{d-n_2 t}{2}+  \lceil \frac{\nu _{2r}}{\nu _{2r} +1}\rceil$ and hence $$a_{r }^{L_{E_2}} (q) \geq a_{r- j(2,0 ) }^{L_{E_2}} (q) = a_{r- j(2,m _2 ) }^{L_{E_2}} (q) +n_2 t  > \frac{d +n_2 t}{2}+  \lceil \frac{\nu _{2r}}{\nu _{2r} +1}\rceil .$$
This combined with the equalities $a_r^{L_{E_1}} (p_1) = d- \nu_{1r} - a_0^{L_{E_1}} (q) = -\nu_{1r} -\beta _0 +  a_r^{L_{E_2}} (q) $ yields that  $a_r^{L_{E_1}} (p_1) > \frac{d+m _2 t}{2} -\beta _0- \nu _{1r} +\lceil \frac{\nu _{2r}}{\nu _{2r} +1}\rceil  .$
Therefore by Lemma \ref{additive}  we obtain the equations:
\begin{equation*}
\begin{cases}
\frac{d+m _2  t}{2}  \leq g_2 +r +\gamma _2 +\eta_{2r} \ \ \ \ \ \ \ \ \ \ \ \ \ \ \ \mbox{ if } a_{j(2,0)} ^{L_{E_2}} (p_2) \geq  a_{r- j(2,m _2) }^{L_{E_2}} (q) -2 \lceil \frac{\nu _{2r}}{\nu _{2r} +1}\rceil  \\
\frac{d+m _2 t}{2} -\beta _0- \nu _{1r} +\lceil \frac{\nu _{2r}}{\nu _{2r} +1}\rceil   < g_1 +r +\gamma _1 +\eta_{10}\  \  \mbox{ if } a_{j(2,0)} ^{L_{E_2}} (p_2) <  a_{r- j(2,m _2) }^{L_{E_2}} (q) -2 \lceil \frac{\nu _{2r}}{\nu _{2r} +1}\rceil  ,
\end{cases}
\end{equation*}
which gives the result (2). Thus the proof of the lemma is completed.
\end{proof}

From Lemma  \ref{inequal} we obtain necessary conditions on $t$  for a TCBE$(g_1, g_2; 2,t)$  curve to  carry a limit linear series $g^r_d$.
\begin{Theorem}\label{mathm}
Let $C$ be a TCBE$(g_1, g_2; 2,t)$  curve of genus $g$ with   $g_1 \geq g_2$ and $t\geq 4$.
If $C$  admits   a limit linear series $g^r_d$ with $\rho (g,r,d)=\rho  <0$, then
\begin{eqnarray}   \label{mathm2} t  < \begin{cases}   g-d+2r  +(g_1 -g_2) +\delta_{g_1, g_2} &\mbox{  in  case  }\rho = -1, \\
\frac{2}{-\rho }\big( \, g-d+2r -2 +(g_1 -g_2) +\delta_{g_1, g_2}\big)
&\mbox{ in  case } \rho \leq -2.
\end{cases}
\end{eqnarray}
\end{Theorem}
\begin{proof}
Assume that $C$ admits a   limit  $g^r_d=\{ N_{Y_1}, N_{E_1}, N_{E_2}, N_{Y_2} \}$ with    $\rho (g, r, d)=\rho $.  Let $\eta _{ij}, \  \beta _{j}, \  \nu _{ij}, \ \gamma _i$ and $m_i$ be the same  as in  Lemma \ref{inequal}.
Since   $(Y_i,p_i)$ is general,  we have $\gamma _i\geq 0$
 for each $ i=1,2.$
  And  Lemma \ref{additive0} tells
\begin{equation*} 
- \rho(L_{E_1}, p_1, q)- \rho(L_{E_2}, p_2, q) = -\rho  +\sum_{i=1}^{2} \sum_{j=0}^{r}  \eta _{ij} + \sum_{j=0}^{r}\beta_j +\gamma_1 +\gamma _2,
\end{equation*}
whence 
\begin{equation} \label{zeros1}
 m _1 +m _2 =  -\rho  +\sum_{i=1}^{2} \sum_{j=0}^{r}  \eta _{ij} + \sum_{j=0}^{r}\beta_j +\gamma_1 +\gamma _2 +\sum_{i=1}^{2}\sum _{j=o} ^r (\nu _{ij} -1)_{+} .
\end{equation}

The remaining parts of the proof will be split into the following two cases:
\begin{equation*}
 (\mbox{Case }1)\  m _1 \leq 0 \ \mbox{ or }\ m _2\leq 0, \ \ \
 (\mbox{Case }2) \ m _1\geq 1 \mbox{ and } m _2\geq 1.
\end{equation*}

\vspace{0.5cm}
{\bf\noindent{(Case 1)}} First we assume $ m _2 \leq 0$.  The equation (\ref{zeros1}) gives
\begin{equation}\label{Delta} m_1\geq -\rho +\sum_{i=1}^{2} \sum_{j=0}^{r}  \eta _{ij} + \sum_{j=0}^{r}\beta_j +\gamma_1 +\gamma _2  +\sum_{i=1}^{2}\sum _{j=o} ^r (\nu _{ij} -1)_{+},
\end{equation}
which combined with  the hypothesis $\rho < 0$ means $m_1 \geq 1$. By 
 Lemma \ref{inequal},(1),   it follows that  either
$$\frac{d+m _1 t}{2} \leq  g_1 +r +\gamma _1 +\eta_{10}$$
or
$$\frac{d+m _1 t}{2} -\beta _r - \nu _{2r}< g_2 +r +\gamma _2 +\eta_{2r}.$$
These combined with  \eqref{Delta} and  the hypothesis $t\geq 4$ yield that either 
\begin{eqnarray*}
2g_1 -d +2r 
&\geq& m _1 t -2(\gamma _1 +\eta_{10} )\\
&\geq& -\rho t + 4(  \gamma _1 +\eta_{10} ) - 2(\gamma _1 +\eta_{10} ) \\
&\geq& -\rho t \, ,
\end{eqnarray*}
or 
\begin{eqnarray*}
2g_2 -d +2r &>& m _1 t -2(\beta _r+\nu _{2r} +\gamma _2 +\eta_{2r})\\
& \geq& -\rho t +4(  \eta _{2r} + \beta_r +\gamma _2  +(\nu _{2r} -1)_{+}) -2(\beta _r+\nu _{2r} +\gamma _2 +\eta_{2r})\\
&\geq& -\rho t -2\, ,
\end{eqnarray*}
since $2(\nu _{2r} -1)_{+} -\nu _{2r}\geq -1$.  Therefore we have 
\begin{equation}\label{eq31}
-\rho t \leq \mbox{max} \{ 2g_1 -d+2r, \ 2g_2 -d+2r +1 \}.
\end{equation}

Assume  $m _1 \leq  0$. From  (\ref{zeros1}) we get 
\begin{equation*}m_2\geq -\rho +\sum_{i=1}^{2} \sum_{j=0}^{r}  \eta _{ij} + \sum_{j=0}^{r}\beta_j +\gamma_1 +\gamma _2  +\sum_{i=1}^{2}\sum _{j=o} ^r (\nu _{ij} -1)_{+} \geq 1.
\end{equation*} Using Lemma \ref{inequal},(2) and 
the same arguments as the above, we obtain that  
\begin{equation}\label{eq32}
    \mbox{  either }\  -\rho t  \leq 2g_2 -d+2r  \  \mbox{  or } \ -\rho t -2 < 2g_1 -d+2r .
  \end{equation}
By \eqref{eq31} and \eqref{eq32} we conclude that if either $ \ m _1 \leq 0  \mbox{ or  } m _2\leq 0,$
\begin{eqnarray}
 \notag t  &\leq&  \frac{1}{-\rho} \big(\, 2g_1 -d+2r  +1 \big) \\
\label{maineq1} &=&\frac{1}{-\rho} \big(  g-d+2r -1+(g_1-g_2)  \big),
\end{eqnarray}
since $2g_1 =g_1 +g_2 +(g_1 -g_2)=g-2 +(g_1 -g_2)$. Thus $t$ satisfies the equation \eqref{mathm2}.   

\vspace{0.5cm}
{\bf\noindent{(Case 2)}} Assume $ m _1\geq 1$ and $ m _2\geq 1$. By Lemma \ref{inequal}, the condition $m_1\geq 1$ implies that 
\begin{equation}\label{*1}\tag{*1}
\frac{d+m _1 t}{2} + \lceil \frac{\nu _{1r}}{\nu _{1r} +1}\rceil\leq g_1 +r +\gamma _1 +\eta_{10}.
\end{equation}
or
\begin{equation}\label{*2}\tag{*2}
\frac{d+m _1 t}{2} -\beta _r- \nu _{2r} +\frac{1}{2} \leq g_2 +r +\gamma _2 +\eta_{2r}.
\end{equation}
and the condition $m_2\geq 1$ implies that either
\begin{equation}\label{**1}\tag{**1}
\frac{d+m _2 t}{2} \leq g_2 +r +\gamma _2 +\eta_{2r}.
\end{equation}
or
\begin{equation}\label{**2}\tag{**2}
\frac{d+m _2 t}{2} -\beta _0- \nu _{1r}+ \lceil \frac{\nu _{2r}}{\nu _{2r} +1}\rceil+\frac{1}{2} \leq g_1 +r +\gamma _1 +\eta_{10}.
\end{equation}
Therefore  we should have one of the following cases:
\begin{eqnarray*}
\begin{cases} g-2 -d+2r \geq \frac{(m _1 +m _2)t}{2} -(\gamma _1 +\eta_{10}+ \gamma _2 +\eta_{2r})
  &\mbox{for } \{ (\ref{*1}),  (\ref{**1}) \}\\
2g_1 -d+2r - \frac{1}{2} \geq \frac{(m _1 +m _2)t}{2} -(2\gamma _1 +2\eta_{10}+ \beta_0 + \nu _{1r}) + \lceil \frac{\nu _{1r}}{\nu _{1r} +1}\rceil
&\mbox{for }  \{(\ref{*1}), (\ref{**2})\} \\
2g_2 -d+2r + \frac{1}{2} \geq \frac{(m _1 +m _2)t}{2} -(2\gamma _2 +2\eta_{2r}+ \beta_r + \nu _{2r})+1
&\mbox{for }  \{(\ref{*2}), (\ref{**1})\} \\
g -2-d+2r \\
~~\geq \frac{(m _1 +m _2)t}{2}
 -(\gamma _1 +\eta_{10}+ \gamma _2 +\eta_{2r}+\beta_0 + \nu _{1r} +\beta_r + \nu _{2r}) + \lceil \frac{\nu _{2r}}{\nu _{2r} +1}\rceil +1
&\mbox{for } \{ (\ref{*2}), (\ref{**2})\}
\end{cases}
\end{eqnarray*}
since $g_1 +g_2 =g-2$.

Since the sum of left hand sides of any  $\{(*k),(**l)\}$ have the term $\frac{(m _1 +m _2)t}{2}$, we consider the following inequalities given by (\ref{zeros1}):
\begin{eqnarray} \label{eq00}
&&\frac{(m _1 +m _2)t}{2} \\
\notag &&\geq  \begin{cases}
\frac{ -\rho t}{2} + \frac{ t}{2}(\sum_{i=1}^{2} \sum_{j=0}^{r}  \eta _{ij} + \sum_{j=0}^{r}\beta_j +\gamma_1 +\gamma _2
 +\sum_{i=1}^{2}\sum _{j=o} ^r (\nu _{ij} -1)_{+})  &\mbox{if } \rho\leq -2 \\
t +   \frac{ t}{2}(-1 + \sum_{i=1}^{2} \sum_{j=0}^{r}  \eta _{ij} + \sum_{j=0}^{r}\beta_j +\gamma_1 +\gamma _2
 +\sum_{i=1}^{2}\sum _{j=o} ^r (\nu _{ij} -1)_{+}) &\mbox{if } \rho =-1,
\end{cases}\\
\notag &&\geq  \begin{cases}
\frac{ -\rho t}{2} + 2(\sum_{i=1}^{2} \sum_{j=0}^{r}  \eta _{ij} + \sum_{j=0}^{r}\beta_j +\gamma_1 +\gamma _2
 +\sum_{i=1}^{2}\sum _{j=o} ^r (\nu _{ij} -1)_{+})  &\mbox{if } \rho\leq -2 \\
t + 2(-1 + \sum_{i=1}^{2} \sum_{j=0}^{r}  \eta _{ij} + \sum_{j=0}^{r}\beta_j +\gamma_1 +\gamma _2
 +\sum_{i=1}^{2}\sum _{j=o} ^r (\nu _{ij} -1)_{+}) &\mbox{if } \rho =-1.
\end{cases}
\end{eqnarray}
Here the last inequality is given by $t\geq 4$ and the inequality $-1+ \sum_{i=1}^{2} \sum_{j=0}^{r}  \eta _{ij} + \sum_{j=0}^{r}\beta_j +\gamma_1 +\gamma _2  +\sum_{i=1}^{2}\sum _{j=o} ^r (\nu _{ij} -1)_{+} \geq 0$ which is derived from  \eqref{zeros1} combined with $m _1 +m _2 \geq 2$ and $\rho =-1$.
Thus the equation \eqref{eq00} yields that for any case  $\{(*k),(**l)\}$
\begin{eqnarray*}
(\mbox{ right hand side of } \{(*k),(**l)\}) \geq \begin{cases} \frac{ -\rho t}{2} &\mbox{ if } \rho\leq -2 \\
t-2  &\mbox{ if } \rho= -1
\end{cases}
\end{eqnarray*}
since $(\nu _{ir}-1)_+ -\nu _{ir} \geq -1$ and $(\nu _{ir}-1)_+ -\nu _{ir} +\lceil \frac{\nu _{ir}}{\nu _{ir} +1}\rceil\geq 0$ for each $i=1,2$.
On the one hand, we have
\begin{eqnarray*}
(\mbox{ left hand side of } \{(*k),(**l) \}) & \leq &
\begin{cases}
2g_1-d+2r+\frac{1}{2}\ \ {\mbox{if $g_1=g_2$}}\\
2g_1-d+2r-\frac{1}{2} \ \ {\mbox{if $g_1>g_2$}}
\end{cases}\\
& =& g-d+2r-\frac{5}{2}+(g_1-g_2)+\delta_{g_1 ,g_2}
\end{eqnarray*}
since $2g_1=g-2+(g_1-g_2)$.   In sum, the conclusion of (Case 2) is  that 
\begin{eqnarray*}   t  \leq \begin{cases}   g-d+2r -\frac{1}{2} +(g_1 -g_2) +\delta_{g_1, g_2} &\mbox{  if  }\rho = -1, \\
\frac{2}{-\rho }\big( \, g-d+2r -\frac{5}{2} +(g_1 -g_2) +\delta_{g_1, g_2}\big)
&\mbox{ if } \rho \leq -2 .
\end{cases}
\end{eqnarray*}
 Therefore we complete the proof of the theorem. 
\end{proof}

Theorem  \ref{mathm} provides  conditions on the torsion $t$ for a TCBE$(g_1, g_2 ; 2,t)$ curve not to carry a  limit $g^r_d$, whereas  Theorem  1.1 in \cite{SK} gives conditions on $t$ for the existence of a smoothable limit $g^r_d$ on the curve. Combining Theorem  1.1 in \cite{SK} and  Theorem \ref{mathm}, we obtain the following corollaries.  

\begin{Corollary}\label{macor}  Let $C$ be a  TCBE$(g_1, g_2; 2,t)$   curve  of genus  $g$ with
 $g_1\geq g_2 $ and let   positive integers $r$ and $d$ satisfy $\rho:=\rho (g, r, d)=-1$ or $-2$.
  \begin{enumerate}
\item $C$ carries a smoothable limit $g^r_d$ if 
\begin{equation}\label{eqexist5}
\begin{cases} t\leq g_2 +3+\rho, \ \ t\equiv g_1 +1 \ (\mbox{mod }2) \ \ \
\mbox{ for } r=1\, \\
  r+3 +(g_1 -g_2) \leq   t\leq g-d+2r+\rho \ \ \ \mbox{ for } r\geq 2  .
\end{cases} 
\end{equation}
\item 
$C$ does not carry a limit $g^{r}_d$ if $t\geq g-d+2r+2(1+\rho)+(g_1-g_2)+\delta_{g_1,g_2}$.
\end{enumerate}
Specifically, in case $\rho=-2$ and $r\geq 2$,  a  TCBE$(\lceil \frac{g-2}{2}\rceil, \lfloor \frac{g-2}{2}\rfloor; 2,t)$ curve with $t\geq r+4$ admits a smoothable limit $g^{r}_d$ if and only if $t\leq g-d+2r-2.$
\end{Corollary}

\begin{proof}   (2) is the result of Theorem \ref{mathm} corresponding to $\rho:=\rho (g, r, d)=-1$ or $-2$.   
To get   (1),  set $$h: = 2+\rho.$$   If  the inequality $g_1 \geq g_2 +h$ is satisfied,   then  
(1)  is the result of Theorem 1.1 in \cite{SK}  corresponding to $n=2$ and $h=0$ or $1$, which is stated in the introduction, since $t_{min} (r; g_1, g_2; 2,h)$ in \eqref{tmin}  is no more than   $r+3 +(g_1 -g_2)\ \mbox{ for } r\geq 2.$  
 Thus it remains to show  $g_1 \geq g_2 +h$.
 In the proof of the theorem the hypothesis $g_1 \geq g_2 +h$ is used only when $r$ is odd. Thus it is suffices to check  $g_1 \geq g_2 +h$ when $r$ is odd and $h=1$. In this case, if $g_1 =g_2$ then 
  $g-\rho =g+1=2g_1 +3 $ is odd, whereas the parity of $g-\rho$ is even since $\rho (g,r,d)=\rho$ means $g-\rho =(r+1)(g-d+r)$ and $r $ is odd. This cannot occur.
Hence we have $g_1 \geq g_2+h$. Therefore  we get  the result (1).

 If  $\rho=-2$ and $r\geq 2$, the results (1) and (2) tell that the inequality $t\leq g-d+2r-2$  becomes a necessary and sufficient condition on $t$ for a   TCBE$(\lceil \frac{g-2}{2}\rceil, \lfloor \frac{g-2}{2}\rfloor; 2,t)$ curve with $t\geq r+4$ to carry a smoothable limit $g^r_d$.
 \end{proof}
Recall that $\pmb{\bigtriangleup}  ^{BE} _g (g_1, g_2; 2;t) : = \{ [C] \in \overline {\mathcal M} _g ~|~ C \mbox{ is a  TCBE}(g_1, g_2; 2,t ) \mbox{ curve, }  g_1\geq g_2\}$.
\begin{Corollary}\label{mcoro4}
Let  $g,r,s,d,e$ be positive integers such that $\rho(g,r,d)=\rho(g,s,e)=-2$, $e\neq d$ and $e\neq 2g-2-d$. Then,
 $$\mbox{supp}({\mathcal M^r_{g,d})}\neq \mbox{supp}({\mathcal M^s_{g,e}}).$$ 
  Further, 
$\pmb{\bigtriangleup}^{BE}_{g} \big(  \lceil \tfrac{g-2}{2}\rceil, \lfloor \tfrac{g-2}{2}\rfloor; 2, g-d+2r-2 \, \big) \subset \, \overline{\mathcal M} ^r_{g,d} -\overline{\mathcal M} ^s_{g,e} ,$
where $d,e\leq g-1$ and $s>r$.
\end{Corollary}
\begin{proof}  Assume that $\rho(g,r,d)=\rho(g,s,e)=-2$, $e\neq d$ and $e\neq 2g-2-d$.                  
In the case of $r=1$,  G. Farkas \cite{Farkas_home} proved   that a general $d$-gonal curve $X$ has no $g^s_e$ with $\rho(g,s,e)<0$ except $g^1_d$ and $K_X  -g^1_d$. 
 Thus we assume that  $s>r\geq 2$ and  that $d,e \leq g-1$ by  Serre duality. 
 
 Let $C$ be a  TCBE$(\lceil \frac{g-2}{2}\rceil, \lfloor \frac{g-2}{2}\rfloor; 2, g-d+2r-2)$ curve.  The proof  
will be given by  verifying the existence (resp. non-existence) of a smoothable $g^r_d$ (resp.  $g^s_e$)    on  $C$.  For the existence of $g^r_d$ on $C$, we will prove that 
\begin{equation*}\label{eqgdr}
 r+4 \leq g-d+2r-2,  \  \mbox{ that is }\   g-d+r\geq 6,
\end{equation*} 
which implies the torsion $t=g-d+2r-2$  is bigger than or equal to the lower bound  in Corollary \ref{macor},(1).
By the hypotheses $\rho(g,r,d)=\rho(g,s,e)=-2$, we have $g+2=(r+1)(g-d+r)=(s+1)(g-e+s)$, which  combined  with $d,e \leq g-1$ and $s>r$ yields
\begin{equation}\label{eqrs}
3 \leq r+1 <s+1 \leq g-e+s <g-d+r,  \ \  s+1\leq \sqrt{g+2}.\end{equation}  Thus we get $g+2\geq  16$.
  If $r=2$,  we have $g-d+r =\frac{g+2}{3} \geq \frac{16}{3}$ and hence $g-d+r\geq 6$ holds.  
  If $r\geq 3$,  the inequality $g-d+r\geq 6$ follows from \eqref{eqrs}. 
  Therefore the curve $C$ admits a smoothable $g^r_d$ by Corollary \ref{macor},(1) for $\rho =-2$. 
 
 It remains to show that $C$ has no smoothable $g^s_e$.
From 
 the equation  $g+2=(r+1)(g-d+r)=(s+1)(g-e+s)$ combined with $r+1 <s+1 \leq \sqrt{g+2}\,$  we  obtain 
 $$(r+1) +(g-d+r) > (s+1)+(g-e+s)$$
 which means $t=g-d+2r -2 >g-e+2s-2$. Thus $C$  has no smoothable $g^s_e$ by Corollary \ref{macor},(2). Therefore there is a smooth curve of genus $g$ in ${\mathcal M^r_{g,d}}$ which has no $g^s_e$, whence $\mbox{supp}({\mathcal M^r_{g,d}})\neq \mbox{supp}({\mathcal M^s_{g,e}})$. 
 \end{proof}

\begin{Corollary}\label{macor2}  A Brill-Noether locus ${\mathcal M ^r_{g,d}}$ with  $\rho (g, r, d)=-2$ is not contained in a Brill-Noether divisor   ${\mathcal M ^s_{g,e}}$
 with   $e-2s \geq d-2r+3$. Specifically,  if  $g\geq 34$,  a Brill-Noether locus 
${\mathcal M ^2_{g,d}}$ with  $\rho (g, 2, d)=-2$ is not contained in a Brill-Noether divisor   ${\mathcal M ^s_{g,e}}$ with $s\geq2$. Further, 
\begin{equation*}
\begin{array}{rl}
&\bullet\ \ \pmb{\bigtriangleup}^{BE}_{g} \big(  \lceil \tfrac{g-2}{2}\rceil, \lfloor \tfrac{g-2}{2}\rfloor; 2, g-d+2r-2 \, \big)  \\
&\qquad\qquad\qquad\subset  \, \overline{\mathcal M} ^r_{g,d} -\bigcup\{ \overline{\mathcal M} ^s_{g,e}~|~\rho (g,s,e)=-1,\  e-2s\geq d-2r+3 \}\\
{\mbox{and}}&\\
&\bullet\ \ \pmb{\bigtriangleup}^{BE}_{g} \big(  \lceil \tfrac{g-2}{2}\rceil, \lfloor \tfrac{g-2}{2}\rfloor; 2, g-d+2\, \big)  \\
& \qquad\qquad\qquad \subset \, \overline{\mathcal M} ^2_{g,d} -\bigcup\{ \overline{\mathcal M} ^s_{g,e}~| ~\rho (g,s,e)=-1,\  s\geq 2\}.
\end{array}
\end{equation*}

\end{Corollary}
\begin{proof} Let $\rho(g,r,d)=-2$ and $\rho (g,s,e)=-1$. By the same reasoning as in the beginning of the proof of  Corollary \ref{mcoro4}, we assume that $r\geq 2$.   Let $C$  be a TCBE$(\lceil \frac{g-2}{2}\rceil, \lfloor \frac{g-2}{2}\rfloor; 2, g-d+2r-2)$ curve.   In the proof of Corollary \ref{mcoro4} we show that   $C$ carries a smoothable $g^r_d$.  On the other hand, by  Corollary \ref{macor},(2),   $C$ has no smoothable $g^s_e$ since the the  hypothesis $e-2s > d-2r+2$ gives $g-d+2r-2 \geq g-e+2s +1$. Therefore, ${\mathcal M ^r_{g,d}}$ is not contained in a Brill-Noether divisor   ${\mathcal M ^s_{g,e}}$.

 Now, consider the case of  ${\mathcal M ^2_{g,d}}$ with $\rho (g,2,d)=-2$ and $g\geq 34$. Let $C$  be a TCBE$(\lceil \frac{g-2}{2}\rceil, $ $\lfloor \frac{g-2}{2}\rfloor; 2, g-d+2)$ curve.  The proof of Corollary \ref{mcoro4} tells that there exists a smoothable $g^2_d$ on $C$ since $g-d+2r-2=g-d+2$ for $r=2$. 
Now, we will show that $C$ does not admit a smoothable $g^s_e$ with $\rho(g,s,e)=-1$.  According to Corollary \ref{macor},(2), we will get the non-existence of $g^s_e$ by proving
\begin{equation}\label{comparet}
g-d+2 \geq g-e+2s+1.
\end{equation}
Foremost  we have $s\neq 2$,  since  two conditions $\rho(g,d,2)=-2$ and $\rho(g,s,e)=-1$ respectively mean
\begin{equation}\label{eqg12}
g+2=3(g-d+2), \ \ \  g+1=(s+1)(g-e+s).
\end{equation}
Assume that $s\geq 4$. Considering $g+1=(s+1)(g-e+s)$ combined with the graph given by $xy=g+1$, we see that $(s+1)+(g-e+s)$ attains the maximum when $s=4$ since $e\leq g-1$. 
From  \eqref{eqg12} it follows 
\begin{equation*}
g-d+2=\frac{g+2}{3}\ \mbox{ and }  \  \  g-e+2s+1=(s+1)+(g-e+s) \leq 5+ \frac{g+1}{5},
\end{equation*}
whence the equation \eqref{comparet} holds for $g\geq 34$. 

 Thus it remains to show \eqref{comparet} in case $s=3$.  The equations in \eqref{eqg12} imply that 
$$g+2\equiv 0\ (\mbox{mod }3), \ \  g+1\equiv 0 \ (\mbox{mod }4),$$
whence $g-7$ is a multiple of  12.  This combined with $g\geq 34$ yields $g\geq 43$.
According to \eqref{eqg12}  we get $g-e+2s +1 = 4 + \frac{g+1}{4}$ which implies the validity of   \eqref{comparet} for $g\geq 43$ since  $g-d+2=\frac{g+2}{3}$.    In sum, we conclude that $C$ does not admit any $g^s_e$ with $\rho (g,s,e)=-1$ and $s\geq 2$. Therefore, we get the results of the corollary.
\end{proof}

\begin{Remark}\label{rmknet}   ${\mathcal M ^2_{g,d}}$ is irreducible  since the Severi variety of degree $d$ and (geometric) genus $g$ plane curves  is irreducible \cite{Har}. In case $\rho (g,2,d) \leq -2$ and $g\geq 34$,  Corollary \ref{mcoro4}   and \ref{macor2} imply that a general curve $X$ having $g^2_d$   has no $g^s_e$ with $s\geq 2$ and $\rho (g,s,e)<0$ except  $ g^2_d$ and   $ K_X -g^2_d$. This means that the smooth model $X$ of   a general plane curve  of degree $d$ and genus $g$ with  $34\leq g\leq \frac{3d-4}{2}$ has  no $g^s_e$   having  $s\geq 2$ and $\rho (g,s,e)<0$ except  $ g^2_d$ and   $ K_{X }-g^2_d$.
\end{Remark}

The following  example shows  relations among Brill-Noether  loci of $\mathcal M_{34}$ .
\begin{Example} \label{exam34} The moduli space $\mathcal M_{34}$ admits one Brill-Noether  divisor $\mathcal M ^4_{34, 31}$ and four Brill-Noether loci with Brill-Noether number $\rho = -2$ as follows:
$$\mathcal M ^1_{34, 17}, \  \mathcal M ^2_{34, 24}, \ \mathcal M ^3_{34, 28} , \ \mathcal M ^5_{34, 33}.$$
Using Corollary \ref{macor} we get some relations among the above Brill-Noether loci.  In Table \ref{table34} and \ref{table34-1}, we precisely demonstrate the  range of $t$ for  the existence and the lower bound  of $t$ for the  nonexistence of a limit linear series $g^r_d$ on a TCBE$(g_1,g_2;2,t)$ curve as follows: in case $r\geq 2$,
\small
\begin{center}
\begin{table}[ht]
\begin{tabular}{c|c|c}
\hline $\ \overline{\mathcal M} ^r_{g,d} $  & $ r+3+(g_1 -g_2) \leq t\leq g-d+2r +\rho $  & $  g-d+2r +2(1+\rho) +(g_1 -g_2) +\delta _{g_1, g_2}$ \\
\hline   $\ \overline{\mathcal M}  ^2_{34, 24}$ & $5 +(g_1 -g_2) \leq t\leq 12 $ & $12+(g_1 -g_2) +\delta _{g_1, g_2}$\\
\hline   $\ \overline{\mathcal M} ^3_{34, 28}$ & $6+(g_1 -g_2)\leq t\leq 10$  & $10+(g_1 -g_2) +\delta _{g_1, g_2}$\\
\hline   $\ \overline{\mathcal M} ^4_{34, 31}$  &$7+(g_1 -g_2) \leq t\leq 10$ & $11+(g_1 -g_2) +\delta _{g_1, g_2}$\\
\hline   $\ \overline{ \mathcal M} ^5_{34, 33}$ & $8+(g_1 -g_2)\leq t\leq 9$ & $9+(g_1 -g_2) +\delta _{g_1, g_2}$ \\
\hline
\end{tabular}
\vspace{0.1cm}
\caption{\label{table34}}
\end{table}
\end{center}
\normalsize
\vspace{-0.5cm}
and  in case $r=1$, 
\small
\begin{center}
\begin{table}[ht]
\begin{tabular}{c|c|c}
\hline $\ \overline{\mathcal M} ^1_{g,d} $  & $ t\leq g-d+3+\frac{3\rho}{2}-\frac{g_1 -g_2}{2}, \ \ t\equiv g_1 +1 \ (\mbox{mod }2) $  & $  g-d +(g_1 -g_2) +\delta _{g_1, g_2}$ \\
\hline   $\ \overline{\mathcal M} ^1_{34,17} $ & $t\leq 17-\frac{g_1 -g_2}{2} ,  \ t\equiv g_1 +1 \ (\mbox{mod }2)  $ & $17+(g_1 -g_2) +\delta _{g_1, g_2}$\\
\hline
\end{tabular}
\vspace{0.1cm}
\caption{\label{table34-1}}
\end{table}
\end{center}
\normalsize
\vspace{-0.5cm}

Corollary \ref{macor} combined with Table \ref{table34} and \ref{table34-1} yields  the following:
\begin{enumerate}
\item[$\bullet$] $\pmb{\bigtriangleup} ^{BE}_{34} (16,16;2,12)   \subset \overline{\mathcal M} ^2_{34, 24}-  ( \overline{\mathcal M} ^3_{34, 28}\cup \overline{\mathcal M} ^4_{34,31} \cup\overline{\mathcal M} ^5_{34, 33} ) $\\
\item[$\bullet$] $\bigcup\limits_{t=11,12}\pmb{\bigtriangleup} ^{BE}_{34} (16,16;2,t) \bigcup \pmb{\bigtriangleup} ^{BE}_{34} (17,15;2,12)  \subset  \overline{\mathcal M} ^2_{34, 24} -(  \overline{\mathcal M} ^3_{34, 28} \cup \overline{\mathcal M}^5_{34, 33}) $\\
\item[$\bullet$] $\bigcup\limits_{t=11,12}\{\pmb{\bigtriangleup} ^{BE}_{34} (16,16;2,t) \bigcup \pmb{\bigtriangleup} ^{BE}_{34} (17,15;2,t) \} \subset  \overline{\mathcal M} ^2_{34, 24} - \overline{\mathcal M}^5_{34, 33} $\\
\item[$\bullet$] $\pmb{\bigtriangleup} ^{BE}_{34} (16,16;2,10)   \subset  (\overline{\mathcal M} ^2_{34, 24} \cap  \overline{\mathcal M} ^3_{34, 28}\cap \overline{\mathcal M} ^4_{34,31}) - \overline{\mathcal M}^5_{34, 33} $\\
\item[$\bullet$] $\pmb{\bigtriangleup} ^{BE}_{34} (16,16;2,11)   \subset  ( \overline{\mathcal M}^1_{34, 17}\cap\overline{\mathcal M} ^2_{34, 24})-  (\overline{\mathcal M} ^3_{34, 28}  
\cup \overline{\mathcal M} ^5_{34, 33})  $\\
\item[$\bullet$] $\bigcup\limits_{t=13,15,17}\pmb{\bigtriangleup} ^{BE}_{34} (16,16;2,t)  \subset  \overline{\mathcal M} ^1_{34, 17} -(\overline{\mathcal M}^2_{34, 24} \cup \overline{\mathcal M} ^3_{34, 28}\cup  \overline{\mathcal M}^5_{34, 33}  \cup \overline{\mathcal M} ^4_{34,31})$\\
\item[$\bullet$] $\pmb{\bigtriangleup} ^{BE}_{34} (16,16;2,9)  \subset  \overline{\mathcal M} ^1_{34, 17} \cap \overline{\mathcal M} ^2_{34, 24} \cap \overline{\mathcal M} ^3_{34, 28} \cap \overline{\mathcal M}^5_{34, 33}  \cap \overline{\mathcal M} ^4_{34, 31}$
\end{enumerate}
 Concerning the last relationship in the above, we make a remark that
 \begin{eqnarray*}
&&\mbox{dim}  \big( \pmb{\partial}\overline{\mathcal M} ^1_{34, 17} \cap  \pmb{\partial}\overline{\mathcal M} ^2_{34, 24} \cap  \pmb{\partial}\overline{\mathcal M} ^3_{34, 28} \cap  \pmb{\partial}\overline{\mathcal M}^5_{34, 33}  \cap  \pmb{\partial}\overline{\mathcal M} ^4_{34, 31}\big)\\
&&\geq \mbox{dim} \big(\pmb{\bigtriangleup} ^{BE}_{34} (16,16;2,9)\big)\\
&&= \mbox{dim}\big(\pmb{\partial}\overline{\mathcal M} ^1_{34, 17} \big)-2\, ,
\end{eqnarray*}
since
$\mbox{dim} \pmb{\bigtriangleup} ^{BE}_{34} (16,16;2,9) = 2\{ \mbox{dim} {\mathcal M}_{16,1} + (\mbox{dim}  {\mathcal M}_{1,2} -1) \}=94.$ 
Here, $\pmb{\partial}\overline{\mathcal M} ^r_{34, d}$ denotes the boundary of  ${\mathcal M} ^r_{34, d}$ in 
 $\overline{\mathcal M} _{34}.$ 
\end{Example}
\pagestyle{myheadings} 

\markboth{}{}

\begin{thebibliography}{plain}
\bibitem{CKK1} Choi, Y., Kim, S. and Kim, Y., Remarks on Brill-Noether Divisors and Hilbert
schemes,  J.  Pure and Appl. Algebra {\bf 216} (2012), 377-384.\

\bibitem{CKK2} Choi, Y., Kim, S. and Kim, Y., Brill-Noether divisors for even genus,  J.  Pure and Appl. Algebra {\bf 218} (2014), 1458-1462.\

\bibitem{Edidin} Edidin, D., Brill-Noether theory in codimension-two, J. Algebraic Geom.  {\bf 2 (1)} (1993), 25-67.\ 

\bibitem{EH1986} Eisenbud, D. and Harris, J., Limit linear series: Basic theory, Invent.
Math. {\bf 85} (1986), 337-371.\

\bibitem{EH1987-2} Eisenbud, D. and Harris, J., The Kodaira dimension of the moduli
space of curves of genus$\ge$23, Invent. Math. {\bf 90} (1987), 359--387.\

\bibitem{EH1989} Eisenbud, D. and Harris, J., Irreducibility of some families of linear series  with Brill-Noether number $-1$, Ann. Scient. \'{E}c. Norm. Sup., $4^c$ s\'{e}rie, t. 22. (1989) 33--53\


\bibitem{Farkas} Farkas, G., The Geometry of the moduli space of curves of genus 23, Math. Ann. {\bf 318} (2000)  43--65.\

\bibitem{Farkas_home} Farkas, G., The birational geometry of the moduli space of curves, Ph. D. Thesis, Universiteit van
Amsterdam, 2000.

\bibitem{Har} Harris, J., On the Severi problem, Inv. Math. {\bf 84} (1986), 445--461.\

\bibitem{SK} Kim, S., Linear series on  a stable curve of compact type and relations among Brill-Noether loci, J.  Alg. {\bf  547} (2020) 70--94.\
\bibitem{St}
 Steffen, F., A generalized principal ideal theorem with an application to
Brill-Noether theory, Invent. Math. {\bf 132} (1998), no.1, 73-89.\

\end{thebibliography}
\end{document}